\documentclass[a4paper,11pt]{article}
\usepackage{amsmath, amsfonts, amscd, amssymb, amsthm, enumerate, xypic}

\def\Mat{\text{M}}
\def\GL{\text{GL}}
\def\bfB{\mathbf{B}}
\def\bfC{\mathbf{C}}
\def\defterm{\textbf}

\def\card{\#\,}
\newcommand{\Ker}{\operatorname{Ker}}
\newcommand{\coker}{\operatorname{coker}}

\newcommand{\im}{\operatorname{Im}}
\renewcommand{\setminus}{\smallsetminus}

\def\K{\mathbb{K}}

\def\N{\mathbb{N}}

\renewcommand{\L}{\mathbb{L}}

\def\lcro{\mathopen{[\![}}
\def\rcro{\mathclose{]\!]}}

\theoremstyle{definition}
\newtheorem{Def}{Definition}
\newtheorem{Not}[Def]{Notation}

\theoremstyle{plain}
\newtheorem{theo}{Theorem}

\newtheorem{lemme}[theo]{Lemma}

\theoremstyle{remark}
\newtheorem{Rems}{Remarks}

\title{Invariance of simultaneous similarity and equivalence of matrices under
extension of the ground field}
\author{Cl\'ement de Seguins Pazzis
\footnote{Teacher at Lyc\'ee Priv\'e Sainte-Genevi\`eve, 2, rue
de l'\'Ecole des Postes, 78029 Versailles Cedex, FRANCE.}
\footnote{e-mail address: dsp.prof@gmail.com}}
\date{\today}
\begin{document}

\maketitle

\abstract{We give a new and elementary proof that simultaneous
similarity and simultaneous equivalence of families of matrices are invariant under extension of the ground field,
a result which is non-trivial for finite fields and first appeared in a
paper of Klinger and Levy (\cite{KlingerLevy}).}

\vskip 2mm
\noindent
\emph{AMS Classification :} 15A21; 12F99

\vskip 2mm
\noindent
\emph{Keywords :} matrices, Kronecker reduction, field extension, simultaneous similarity, simultaneous equivalence.

\section{Introduction}

In this article, we let $\K$ denote a field, $\L$ a field extension of $\K$, and
$n$ and $p$ two positive integers.

\begin{Def}
Two families $(A_i)_{i \in I}$ and $(B_i)_{i \in I}$ of
matrices of $\Mat_n(\K)$ indexed over the same set $I$ are said to be \defterm{simultaneously similar}
when there exists $P \in \GL_n(\K)$ such that
$$\forall i \in I, \; P\,A_i\,P^{-1}=B_i$$
(such a matrix $P$ will then be called a \defterm{base change matrix}
with respect to the two families).

\item Two families $(A_i)_{i \in I}$ and $(B_i)_{i \in I}$ of
matrices of $\Mat_{n,p}(\K)$ indexed over the same set $I$ are said to be \defterm{simultaneously equivalent}
when there exists a pair $(P,Q) \in \GL_n(\K) \times \GL_p(\K)$ such that
$$\forall i \in I, \; P\,A_i\,Q=B_i.$$
\end{Def}

\noindent Of course, those relations extend the familiar relations of similarity and equivalence
respectively on $\Mat_n(\K)$ dans $\Mat_{n,p}(\K)$, and they are equivalence
relations respectively on $\Mat_n(\K)^I$ dans $\Mat_{n,p}(\K)^I$.

\noindent The simultaneous similarity of matrices is generally regarded upon as a ``wild problem"
where finding a useful characterisation by invariants seems out of reach.
See \cite{Friedland} for an account of the problem and an algorithmic approach to its solution
(for that last matter, also see \cite{KlingerLevy}).

\vskip 2mm
\noindent In this respect, our very limited goal here is to establish the following two results :

\begin{theo}\label{theosim}
Let $\K-\L$ be a field extension and $I$ be a set. \\
Let $(A_i)_{i \in I}$ and $(B_i)_{i \in I}$ be two families of matrices of $\Mat_n(\K)$. \\
Then $(A_i)_{i \in I}$ and $(B_i)_{i \in I}$ are simultaneously similar in $\Mat_n(\K)$
if and only if they are simultaneously similar in $\Mat_n(\L)$.
\end{theo}

\begin{theo}\label{theoequiv}
Let $\K-\L$ be a field extension and $I$ be a set. \\
Let $(A_i)_{i \in I}$ and $(B_i)_{i \in I}$ be two families of matrices of $\Mat_{n,p}(\K)$. \\
Then $(A_i)_{i \in I}$ and $(B_i)_{i \in I}$ are simultaneously equivalent in $\Mat_{n,p}(\K)$
if and only if they are simultaneously equivalent in $\Mat_{n,p}(\L)$.
\end{theo}

\begin{Rems} ${}$
\begin{enumerate}[(i)]
\item In both theorems, the ``only if" part is trivial.
\item It is an easy exercise to derive theorem \ref{theosim} from theorem \ref{theoequiv}.
However, we will do precisely the opposite !
\end{enumerate}
\end{Rems}

\section{A proof for simultaneous similarity}

\subsection{A reduction to special cases}

In order to prove theorem \ref{theoequiv}, we will not, \emph{contra} \cite{KlingerLevy},
try to give a canonical form for simultaneous similarity.
Instead, we will focus on base change matrices and
prove directly that if one exists in $\Mat_n(\L)$, then another (possibly the same),
also exists in $\Mat_n(\K)$. To achieve this, we will prove the theorem
in the two following special cases:
\begin{enumerate}[(i)]
\item $\K$ has at least $n$ elements;
\item $\K-\L$ is a separable quadratic extension.
\end{enumerate}
Assuming these cases have been solved, let us immediately prove the general case.
Case (i) handles the situation where $\K$ is infinite. Assume now that
$\K$ is finite, and choose a positive integer $N$ such that $(\card \K)^{2^N}\geq n$. \\
Since $\K$ is finite, there exists (see section V.4 of \cite{Lang}) a tower of $N$ quadratic separable extensions
$$\K \subset K_1 \subset K_2 \subset \cdots \subset K_N.$$
We let $\mathbb{M}$ denote a compositum extension of $K_N$ and $\L$ (as extensions of $\K$) :
$$\xymatrix{
\K \ar@{-}[d] \ar@{-}[r] & K_1 \ar@{-}[r] & K_2 \ar@{-}[r] & \cdots \ar@{-}[r] & K_N \ar@{-}[d] \\
\L \ar@{-}[rrrr] & & & & \mathbb{M}.}$$
Assume the families $(A_i)_{i \in I}$ and $(B_i)_{i \in I}$
of matrices of $\Mat_n(\K)$ are simultaneously similar in $\Mat_n(\L)$.
Then they are also simultaneously similar in $\Mat_n(\mathbb{M})$.
However, $\card K_N=(\card \K)^{2^N} \geq n$, so this simultaneous similarity
also holds in $\Mat_n(K_N)$.
Using case (ii) by induction, when then obtain that
that $(A_i)_{i \in I}$ and $(B_i)_{i \in I}$ are simultaneously similar in $\Mat_n(\K)$.

\subsection{The case $\#\,\K \geq n$}

The line of reasoning here is folklore, but we reproduce the proof for sake of completeness.
Let then $P \in \GL_n(\L)$ be such that
$$\forall i \in I, \; P\,A_i\,P^{-1}=B_i,$$
so
$$\forall i \in I, \; P\,A_i=B_i\,P.$$
Let $V$ denote the $\K$-vector subspace of $\L$ generated by the coefficients of $P$,
and choose a basis $(x_1,\dots,x_N)$ of $V$. Decompose then
$$P=x_1\,P_1+\dots+x_N\,P_N$$
with $P_1,\dots,P_N$ in $\Mat_n(\K)$, and let $W$ be the $\K$-vector subspace
of $\Mat_n(\K)$ generated by the $N$-tuple $(P_1,\dots,P_N)$.
Since the $A_i$'s and the $B_i$'s have all their coefficients in $\K$, the previous relations
give :
$$\forall i \in I, \; \forall k \in \lcro 1,N\rcro, \; P_k\,A_i=B_i\,P_k$$
hence
$$\forall i \in I, \; \forall Q \in W, \; Q\, A_i=B_i\, Q.$$
It thus suffices to prove that $W$ contains a non-singular matrix. \\
However, the polynomial $\det(Y_1\,P_1+\dots+Y_N\,P_N) \in \K[Y_1,\dots,Y_N]$
is homogeneous of total degree $n$ and is not the zero polynomial
because
$$\det(x_1.P_1+\dots+x_N.P_N)=\det(P) \neq 0.$$
Since $n \leq \card \K$, we conclude that the map $Q \mapsto \det Q$ does not totally vanish on $W$, which proves that
$W \cap \GL_n(\K)$ is non-empty, QED.

\subsection{The case $\L$ is a separable quadratic extension of $\K$}

We choose an arbitrary element $\varepsilon \in \L \setminus \K$
and let $\sigma$ denote the non-identity automorphism of the $\K$-algebra $\L$.
Assume $(A_i)_{i \in I}$ and $(B_i)_{i \in I}$ are simultaneously similar in $\Mat_n(\L)$, and
let $P \in \GL_n(\L)$ be such that
$$\forall i \in I, \; P\,A_i\,P^{-1}=B_i.$$
We first point out that the problem is essentially unchanged should $P$ be replaced with a
$\K$-equivalent matrix of $\GL_n(\L)$. \\
Indeed, let $(P_1,P_2)\in \GL_n(\K)^2$, and set $P':=P_1\,P\,P_2^{-1} \in \GL_n(\L)$,
and $A'_i:=P_2\,A_i\,(P_2)^{-1}$ and $B'_i:=P_1\,B_i\,(P_1)^{-1}$ for all $i \in I$.
Then :
$$\forall i \in I, \; P' \,A'_i\, (P')^{-1}=B'_i.$$
Since it follows directly from definition that
$(A_i)_{i \in I}$ and $(A'_i)_{i \in I}$ are simultaneously similar in $\Mat_n(\K)$,
and that it is also true of $(B_i)_{i \in I}$ and $(B'_i)_{i \in I}$,
it will suffice to show that $(A'_i)_{i \in I}$ and $(B'_i)_{i \in I}$ are simultaneously similar in $\Mat_n(\K)$,
knowing that they are simultaneously similar in $\Mat_n(\L)$.

\vskip 2mm
\noindent Returning to $P$, we split it as
$$P=Q+\varepsilon\,R \qquad \text{with $(Q,R)\in \Mat_n(\K)^2$.}$$
The previous remark then reduces the proof to the case where the pair
$(Q,R)$ is canonical in terms of Kronecker reduction
(see chapter XII of \cite{Gantmacher} and our section \ref{appendix}). More roughly, when can assume, since
$P$ is non-singular, that, for some $q \in \lcro 0,n\rcro$:
$$Q=\begin{bmatrix}
M & 0 \\
0 & I_{n-q}
\end{bmatrix} \quad \text{and} \quad
R=\begin{bmatrix}
I_q & 0 \\
0 & N
\end{bmatrix}$$
where $M \in \Mat_q(\K)$, $N$ is a nilpotent matrix of $\Mat_{n-q}(\K)$, and we have let $I_k$
denote the unit matrix of $\Mat_k(\K)$.

\vskip 2mm
\noindent
Let $i \in I$. Applying $\sigma$ coefficient-wise to
$P\,A_i\,P^{-1}=B_i$, we get:
$$\sigma(P)\,A_i\,\sigma(P)^{-1}=B_i=P\,A_i\,P^{-1},$$
hence $A_i$ commutes with $\sigma(P)^{-1}\,P$.
We now claim the following result:

\begin{lemme}\label{lastlemma}
Under the preceding assumptions, any matrix of $\Mat_n(\K)$ that commutes with $\sigma(P)^{-1}\,P$
also commutes with $P$.
\end{lemme}

\noindent Assuming this lemma holds, we deduce that $\forall i \in I, \; P\,A_i\,P^{-1}=A_i$, hence
$(A_i)_{i \in I}$ and $(B_i)_{i \in I}$ are equal, thus simultaneously similar in $\Mat_n(\K)$,
which finishes our proof.

\begin{proof}[Proof of lemma \ref{lastlemma}]
Let $A \in \Mat_n(\K)$ which commutes with $\sigma(P)^{-1}\,P$.
Applying $\sigma$, we deduce that $A$ also commutes with
$P^{-1}\sigma(P)$, hence with $I_n+(\sigma(\varepsilon)-\varepsilon)\,P^{-1}\,R$,
hence with $P^{-1}\,R$ since $\sigma(\varepsilon) \neq \varepsilon$. \\
Notice then that
$$P^{-1}\,R=\begin{bmatrix}
(M+\varepsilon.I_q)^{-1} & 0 \\
0 & (I_{n-q}+\varepsilon\,N)^{-1}\, N
\end{bmatrix}$$
with $(M+\varepsilon.I_q)^{-1}$ non-singular and $(I_n+\varepsilon\,N)^{-1} N$ nilpotent, so
$A$, which stabilizes both $\im (P^{-1}\,R)^n$ and $\Ker (P^{-1}\,R)^n$, must be of the form
$$A=\begin{bmatrix}
C & 0 \\
0 & D
\end{bmatrix} \quad \text{for some $(C,D)\in \Mat_q(\K) \times \Mat_{n-q}(\K)$.}$$
Commutation of $A$ with $P^{-1}\,R$ ensures that $C$ commutes with $(M+\varepsilon.I_q)^{-1}$,
 whereas $D$ commutes with $(I_{n-q}+\varepsilon\,N)^{-1} N=\varepsilon^{-1}.I_{n-q}-
 \varepsilon^{-1}.(I_{n-q}+\varepsilon\,N)^{-1}$
hence with $(I_{n-q}+\varepsilon\,N)^{-1}$.
It follows that $A$ commutes with $P^{-1}$, hence with $P$.
\end{proof}

\section{A proof for simultaneous equivalence}

We will now derive theorem \ref{theoequiv} from theorem \ref{theosim}.
Under the assumptions of theorem \ref{theoequiv}, we choose an arbitrary object
$a$ that does not belong to $I$, and
define
$$C_a=D_a:=\begin{bmatrix}
I_n & 0 \\
0 & 0
\end{bmatrix} \in \Mat_{n+p}(\K)$$
and, for $i \in I$,
$$C_i=\begin{bmatrix}
0 & A_i \\
0 & 0
\end{bmatrix} \quad \text{and} \quad D_i=\begin{bmatrix}
0 & B_i \\
0 & 0
\end{bmatrix} \quad \text{in $\Mat_{n+p}(\K)$.}$$
The following two conditions are then equivalent :
\begin{enumerate}[(i)]
\item $(A_i)_{i \in I}$ and $(B_i)_{i \in I}$ are simultaneously equivalent ;
\item $(C_i)_{i \in I \cup \{a\}}$ and $(D_i)_{i \in I \cup \{a\}}$ are simultaneously similar.
\end{enumerate}
Indeed, if condition (i) holds, then we choose $(P,Q)\in \GL_n(\K) \times \GL_p(\K)$
such that $\forall i \in I, \; P\,A_i\,Q=B_i$, set $R:=\begin{bmatrix}
P & 0 \\
0 & Q^{-1}
\end{bmatrix}$, and remark that $R \in \GL_{n+p}(\K)$ and
$$\forall i \in I \cup \{a\}, \; R\,C_i\,R^{-1}=D_i.$$
Conversely, assume condition (ii) holds, and choose $R \in \GL_{n+p}(\K)$ such that
$$\forall i \in I \cup \{a\}, \; R\,C_i\,R^{-1}=D_i.$$
Equality $R\,C_a\,R^{-1}=C_a$ then entails that $R$ is of the form
$$R=\begin{bmatrix}
P & 0 \\
0 & Q
\end{bmatrix} \quad \text{for some $(P,Q)\in \GL_n(\K) \times \GL_p(\K)$,}$$
and the other relations then imply that
$$\forall i \in I, \; P\,A_i\,Q^{-1}=B_i.$$
Using equivalence of (i) and (ii) with both fields $\K$ and $\L$,
theorem \ref{theoequiv} follows easily from theorem \ref{theosim}.

\section{Appendix : on the Kronecker reduction of matrix pencils}\label{appendix}

Attention was brought to me that, in \cite{Gantmacher}, the proof that
every pencil of matrix is equivalent to a canonical one fails for finite fields. We will give
a correct proof here in the case of a ``weak" canonical form (that is all we need here, and
reducing further to a true canonical form is not hard from there using the theory of elementary divisors).

\begin{Not}
For $n \in \N$,
set $L_n=\begin{bmatrix}
1 & 0 & 0 & &  \\
0 & 1 & 0 & &  \\
 & & \ddots & \ddots &  \\
 &        &  & 1 & 0
\end{bmatrix} \in \Mat_{n,n+1}(\K)$
and $K_n=\begin{bmatrix}
0 & 1 & 0 & &  \\
0 & 0 & 1 & &  \\
 & & \ddots & \ddots & \\
 &         &  & 0 & 1
\end{bmatrix} \in \Mat_{n,n+1}(\K)$;
and, for arbitrary objects $a$ and $b$, define the Jordan matrix:
$$J_n(a,b)=\begin{bmatrix}
a & b & 0 &  \\
0 & a & b &  \\
 & & \ddots & \ddots   \\
\end{bmatrix} \in \Mat_n(\{0,a,b\}).$$
\end{Not}

\begin{theo}[Kronecker reduction theorem for pencils of matrices]
Let $A$ and $B$ in $\Mat_{n,p}(\K)$.
Then there are non-singular $(P_1,Q_1)\in \GL_n(\K) \times \GL_p(\K)$ such that
$P_1\,(A+X\,B)\,Q_1$ is block-diagonal with every non-zero diagonal block having one of the following forms, with only one
of the first type:
\begin{itemize}
\item $P+X\,I_r$ for some non-singular $P \in \GL_r(\K)$;
\item $J_r(1,X)$; \hskip 4mm $J_r(X,1)$; \hskip 4mm  $L_r+X K_r $; \hskip 4mm  $(L_r+XK_r)^t$.
\end{itemize}
This decomposition is unique up to permutation of blocks and up to similarity on the non-singular $P$.
\end{theo}

We will only prove here that such a decomposition exists. Uniqueness is not needed here so we will
leave it as an exercise for the reader. \\
We will consider $A$ and $B$ as linear maps from $E=\K^p$ to $F=\K^n$.
Without loss of generality, we may assume $\Ker A \,\cap\, \Ker B=\{0\}$ and $\im A +\im B =F$.
We define inductively two towers $(E_k)_{k \in \N}$ and $(F_k)_{k \in \N}$ of linear subspaces of $E$ and $F$ by:
\begin{enumerate}[(a)]
\item $E_0=\{0\}$ ; $F_0=A(\{0\})=\{0\}$ ;
\item $\forall k \in \N, \; E_{k+1}=B^{-1}(F_k) \; \text{and}\; F_{k+1}=A(E_{k+1})$.
\end{enumerate}
Notice that $E_1=\Ker B$. The sequences $(E_k)_{n \geq 0}$ and $(F_k)_{n \geq 0}$ are clearly non-decreasing so we can find
a smallest integer $N$ such that $E_N=E_k$ for every $k \geq N$.
Hence $F_N=F_k$ for every $k \geq N$, and $E_N=g^{-1}(F_N)$.
It follows that $A(E_N)=F_N$ and $B(E_N) \subset F_N$.
We now let $f$ and $g$ denote the linear maps from $E_N$ to $F_N$ induced by $A$ and $B$.

\noindent From there, the proof has two independent major steps:

\begin{lemme}\label{choixdebase}
There are basis $\bfB$ and $\bfC$ respectively of $E_N$ and $F_N$ such that
$M_{\bfB,\bfC}(f)+X\,M_{\bfB,\bfC}(g)$ is block-diagonal with all non-zero blocks having one of the
forms $J_r(1,X)$ or $L_s+X\,K_s$.
\end{lemme}

\begin{lemme}\label{bonsupplementaire}
There are splittings $E=E_N\oplus E'$ and $F=F_N \oplus F'$ such that $A(E') \subset F'$ and $B(E') \subset F'$.
\end{lemme}

\noindent Assuming those lemmas are proven, let us see how we can easily conclude:
\begin{itemize}
\item We deduce from the two previous lemmas that $A+X\,B$ is $\K$-equivalent to some
$\begin{bmatrix}
A'+X\,B' & 0 \\
0 & C(X)
\end{bmatrix}$ where $C(X)$ is block-diagonal with all non-zero blocks
of the form $J_r(1,X)$ or $L_s+X\,K_s$, and $A'$ and $B'$ have coefficients in $\K$,
with $\Ker B'=\{0\}$; it will thus suffice to prove the existence of
a canonical form for the pair $(A',B')$;
\item applying the first step of the proof to the matrices $(A')^t$ and $(B')^t$,
we find that $A'+X\,B'$ is $\K$-equivalent to some
$\begin{bmatrix}
A''+X\,B'' & 0 \\
0 & D(X)
\end{bmatrix}$ where $D(X)$ is block-diagonal with all non-zero blocks
of the form $J_r(1,X)^t$ (which is $\K$-similar to $J_r(1,X)$) or $(L_s+X\,K_s)^t$, and $A''$ and $B''$ have coefficients in $\K$,
with $\Ker B''=\{0\}$ and $\coker B''=\{0\}$. It follows that
$B''$ is non-singular.
\item Finally, $(B'')^{-1}(A''+X\,B'')=(B'')^{-1}A''+X.I_k$ for some integer $k$,
and the pair $(A'',B'')$ can thus be reduced by using the Fitting decomposition of
$(B'')^{-1}A''$ combined with a Jordan reduction of its nilpotent part: this yields a block-diagonal matrix
$\K$-equivalent to $A''+X\,B''$ with all diagonal blocks of the form $J_r(X,1)$ or $P+X.I_s$ for some non-singular $P$.
This completes the proof of existence.
\end{itemize}

\begin{proof}[Proof of lemma \ref{bonsupplementaire}]
We proceed by induction. \\
Assume, for some $k \in \lcro 1,N\rcro$, that there are splittings $E=E_N\oplus E'$ and $F=F_N \oplus F'$ such that
$A(E') \subset F' \oplus F_k$ and $B(E') \subset F' \oplus F_k$.
Since $B^{-1}(F_N)=E_N$, the subspaces $F_N$ and $B(E')$ are independant. We can therefore
find some $F''$ such that $F'\oplus F_k=F''\oplus F_k$, $F_N \oplus F''=F$ and $B(E') \subset F''$.
Choose then a basis $(e_1,\dots,e_p)$ of $E'$, and decompose
$A(e_i)=f_i+f'_i$ for all $i \in \lcro 1,p\rcro$, with $f_i \in F''$ and $f'_i \in F_k$.
For $i \in \lcro 1,p\rcro$, we have $f'_i=A(g_i)$ for some $g_i \in E_k$.
Then $(e_1-g_1,\dots,e_p-g_p)$ still generates a supplementary subspace $E''$ of $E_N$ in $E$,
and we now have $A(e_i-g_i) \in F''$ and $B(e_i-g_i) \in F'' \oplus F_{k-1}$ for all $i \in \lcro 1,p\rcro$.
Hence $E=E_N\oplus E''$ and $F=F_N \oplus F''$, now with
$A(E'') \subset F'' \oplus F_{k-1}$ and $B(E'') \subset F'' \oplus F_{k-1}$.
The condition is thus proven at the integer $k-1$. By downward induction, we find that it holds for $k=0$, QED.
\end{proof}

\begin{proof}[Proof of lemma \ref{choixdebase}]
The argument is similar to the standard proof of the Jordan reduction theorem.
\begin{itemize}
\item Split $F_N=F_{N-1} \oplus W_{N,N}$ and
$E_N=E_{N-1} \oplus V_{N,N} \oplus V'_{N,N}$ such that $E_{N-1}\oplus V'_{N,N}=E_{N-1}+(E_N \cap \Ker f)$,
$V'_{N,N} \subset \Ker f$ and $f(V_{N,N})=W_{N,N}$ (so $f$ induces an isomorphism from $V_{N,N}$ to $W_{N,N}$).
Set $W_{N,N-1}=g(V_{N,N})$ and $W'_{N,N-1}=g(V'_{N,N})$.
Remark that $F_{N-2} \oplus W_{N,N-1} \oplus W'_{N,N-1} \subset F_{N-1}$, and split
$F_{N-1}=F_{N-2} \oplus W_{N,N-1} \oplus W'_{N,N-1} \oplus W_{N-1,N-1}$.
\item We then proceed by downward induction to define four families of linear subspaces
$(V_{\ell,k})_{1 \leq k \leq \ell \leq N}$, $(V'_{\ell,k})_{1 \leq k \leq \ell \leq N}$
$(W_{\ell,k})_{1 \leq k \leq \ell \leq N}$ and $(W'_{\ell,k})_{1 \leq k \leq \ell-1 \leq N-1}$ such that:
\begin{enumerate}[(i)]
\item for every $k \in \lcro 1,N\rcro$,
$$E_k=E_{k-1}\oplus V_{k,k} \oplus V_{k+1,k} \oplus \cdots \oplus V_{N,k} \oplus
V'_{k,k} \oplus V'_{k+1,k} \oplus \cdots \oplus V'_{N,k};$$
\item for every $k \in \lcro 1,N\rcro$,
$$F_k=F_{k-1}\oplus W_{k,k} \oplus W_{k+1,k} \oplus \cdots \oplus W_{N,k} \oplus
W'_{k+1,k} \oplus W'_{k+2,k} \oplus \cdots \oplus W'_{N,k};$$
\item for every $k \in \lcro 1,N\rcro$, $E_{k-1}+(E_k \cap \Ker f)=E_{k-1}\oplus V'_{k,k}$
and $V'_{k,k} \subset \Ker f$;
\item for every $\ell \in \lcro 1,N\rcro$ and $k \in \lcro 2,\ell\rcro$,
$g$ induces an isomorphism $g_{\ell,k}: V_{\ell,k} \overset{\simeq}{\longrightarrow} W_{\ell,k-1}$ and an isomorphism
$g'_{\ell,k} : V'_{\ell,k} \overset{\simeq}{\longrightarrow} W'_{\ell,k-1}$;
\item for every $\ell \in \lcro 1,N\rcro$ and $k \in \lcro 1,\ell\rcro$,
$f$ induces an isomorphism $f_{\ell,k} : V_{\ell,k} \overset{\simeq}{\longrightarrow} W_{\ell,k}$ and, if $k<\ell$,
an isomorphism $f'_{\ell,k} : V'_{\ell,k} \overset{\simeq}{\longrightarrow} W'_{\ell,k.}$
\end{enumerate}

$$\xymatrix{
 & V_{\ell,1} \ar[dl]_g \ar[d]^{\simeq}_f & \ar[dl]_g & \cdots & V_{\ell,\ell-1} \ar[dl]_g^{\simeq} \ar[d]^{\simeq}_f & V_{\ell,\ell}  \ar[dl]_g^{\simeq} \ar[d]^{\simeq}_f \\
 \{0\} & W_{\ell,1} & \cdots & & W_{\ell,\ell-1} & W_{\ell,\ell.}
}$$

$$\xymatrix{
 & V'_{\ell,1} \ar[dl]_g \ar[d]^{\simeq}_f & \ar[dl]_g & \cdots & V'_{\ell,\ell-1} \ar[dl]_g^{\simeq} \ar[d]^{\simeq}_f & V'_{\ell,\ell}  \ar[dl]_g^{\simeq} \ar[d]_f \\
 \{0\} & W'_{\ell,1} & \cdots & & W'_{\ell,\ell-1} & \{0\}.
}$$

\item Set $\ell \in \lcro 1,N\rcro$. Define
$$G_\ell=V_{\ell,1}\oplus \dots \oplus V_{\ell,\ell,} \quad
G'_\ell=V'_{\ell,1}\oplus \dots \oplus V'_{\ell,\ell,}$$
$$H_\ell=W_{\ell,1}\oplus \dots \oplus W_{\ell,\ell} \quad \text{and}\quad H'_\ell=W'_{\ell,1}\oplus \dots \oplus W'_{\ell,\ell-1.}$$
Notice that:
$$f(G_\ell)=H_\ell, \quad g(G_\ell)\oplus W_{\ell,\ell}=H_\ell,\quad  f(G'_\ell)=H'_\ell \quad \text{and}\quad g(G'_\ell)=H'_\ell.$$
From there, it is easy to conclude.
\item Let $n_\ell=\dim W_{\ell,\ell}$. Remark that $\dim V_{\ell,k}=\dim W_{\ell,k}=n_\ell$ for every
$1 \in \lcro 1,\ell\rcro$ and choose a basis $\bfC_{\ell,\ell}$ of $W_{\ell,\ell}$.
Define $\bfB_{\ell,\ell}=f_{\ell,\ell}^{-1}(\bfC_{\ell,\ell})$,
$\bfC_{\ell,\ell-1}:=g_{\ell,\ell}(\bfB_{\ell,\ell})$ and proceed by induction to recover a basis
for $V_{\ell,k}$ and $W_{\ell,k}$ for every suitable $k$:
by glueing together those basis, we recover respective basis
 $(\bfB_{\ell,1},\dots,\bfB_{\ell,\ell})$ and $(\bfC_{\ell,1},\dots,\bfC_{\ell,\ell})$ of
 $G_\ell$ and $H_\ell$ and remark that $f$ and $g$ induce linear maps from $G_\ell$ to $H_\ell$ with respective matrices $L_\ell \otimes I_{n_\ell}$ and $K_\ell \otimes I_{n_\ell}$ in those basis (remember that $E_1=\Ker g$).
A simple permutation of basis shows that those linear maps can be represented by $I_{n_\ell} \otimes L_\ell$ and $I_{n_\ell} \otimes K_\ell$
in a suitable common pair of basis.
\item Proceeding similarly for $G'_\ell$ and $H'_\ell$, but starting from a basis of $V'_{\ell,\ell}$, we obtain that
$f$ and $g$ induce linear maps from $G'_\ell$ to $H'_\ell$ and there is a suitable choice of basis so that
their matrices are respectively $I_s \otimes I_\ell$ and $I_s \otimes J_{\ell}(0,1)$ for some integer $s$.
\item Notice that we have defined splittings
$$E_N=G_1\oplus G'_1 \oplus G_2 \oplus G'_2 \oplus \cdots \oplus G_N \oplus G'_N$$
and
$$F_N=H_1\oplus H'_1 \oplus H_2 \oplus H'_2 \oplus \cdots \oplus H'_{N-1} \oplus H_N,$$
therefore lemma \ref{choixdebase} is proven by glueing together the various basis built here.
\end{itemize}
\end{proof}

\end{document}